\documentclass{amsart}
\usepackage{amssymb}

\usepackage{amscd}
\usepackage{thmdefs}

\input{tcilatex}
\theoremstyle{definition}
\theoremstyle{remark}
\numberwithin{equation}{section}

\input tcilatex

\begin{document}
\title[points of continuity of quasiconvex functions]{points of continuity of quasiconvex functions on topological vector spaces}
\author{Patrick J. Rabier}
\address{Department of mathematics, University of Pittsburgh, Pittsburgh, PA 15260}
\email{rabier@imap.pitt.edu}
\maketitle

\begin{abstract}
We give necessary and sufficient conditions for a real-valued quasiconvex
function $f$ on a Baire topological vector space $X$ (in particular, Banach
or Fr\'{e}chet space) to be continuous at the points of a residual subset of 
$X.$ These conditions involve only simple topological properties of the
lower level sets of $f.$ A main ingredient consists in taking advantage of a
remarkable property of quasiconvex functions relative to a topological
variant of essential extrema on the open subsets of $X.$ One application is
that if $f$ is quasiconvex and continuous at the points of a residual subset
of $X,$ then with a single possible exception, $f^{-1}(\alpha )$ is nowhere
dense or has nonempty interior, as is the case for everywhere continuous
functions. As a barely off-key complement, we also prove that every usc
quasiconvex function is quasicontinuous in the (topological) sense of
Kempisty since this interesting property does not seem to have been noticed
before.
\end{abstract}

\section{Introduction\label{intro}}

If $X$ is a vector space, a function $f:X\rightarrow \overline{\Bbb{R}}$ is
quasiconvex if its lower level sets 
\begin{equation}
F_{\alpha }:=\{x\in X:f(x)<\alpha \},  \label{1}
\end{equation}
are convex for every $\alpha \in \Bbb{R}.$ This is equivalent to the
convexity of the level sets 
\begin{equation}
F_{\alpha }^{\prime }:=\{x\in X:f(x)\leq \alpha \}  \label{2}
\end{equation}
and also equivalent to assuming that $f(\lambda x+(1-\lambda )y)\leq \max
\{f(x),f(y)\}$ for every $x,y\in X$ and $\lambda \in [0,1].$

It is a well known result of Crouzeix \cite{Cr81} (see also \cite{ChCr87})
that a real-valued (i.e., finite) quasiconvex function $f$ on $\Bbb{R}^{N}$
is a.e. Fr\'{e}chet differentiable and therefore a.e. continuous. That no
property of this sort can generally be true when $X$ is an infinite
dimensional topological vector space (tvs) follows at once from the
existence of discontinuous linear forms on $X,$ which are nowhere continuous.

The features of the linear case extend to convex functions, which can only
be nowhere or everywhere continuous, with the latter happening if and only
if the function is bounded above on some nonempty open subset (\cite[p. 92]
{Bo53}, \cite{GeKo89}). Little seems to be known about the continuity of
quasiconvex functions when $\dim X=\infty .$ We are only aware of an
unpublished report of Hadjisavvas \cite{Ha94}, where another finite
dimensional result of Crouzeix \cite{Cr81b} is generalized; see \cite{AuDa00}
for more details.

The main goal of this paper is to show that in spite of sharp differences
with the finite dimensional or convex case, there are still simple necessary
and sufficient conditions for a real-valued quasiconvex function on a Baire
tvs $X$ to be continuous at the points of a residual (and therefore dense)
subset of $X,$ i.e., discontinuous only at the points of a subset of (Baire)
first category in $X.$ These conditions involve only basic topological
properties of the sublevel sets (\ref{1}) and (\ref{2}).

Above, ``Baire tvs'' refers to the fact that $X$ is a tvs and that no
nonempty open subset of $X$ can be covered by countably many closed subsets
with empty interior. Equivalently, no nonempty open subset of $X$ is of
first category. First and second categories are always understood relative
to the whole space $X.$

Since every complete metric space is a Baire space, the class of Baire tvs
includes Banach or Fr\'{e}chet spaces, but there are many other
non-metrizable examples in the literature. Local convexity is not required
but significant applications may be limited without it. For example, if $%
X=L^{p}(0,1)$ with $0<p<1,$ the only open convex subsets are $X$ and $%
\emptyset $ \cite[p. 116]{Co85}, so that the quasiconvex functions on $X$
continuous at the points of a residual subset are actually constant on a
residual subset. This follows at once from the criterion in the next
paragraph.

One of the main results, given in Corollary \ref{cor10} along with other
equivalent but slightly more technical statements, is that the set of points
of discontinuity of a quasiconvex function $f:X\rightarrow \Bbb{R}$ is of
first category if and only if, with the notations (\ref{1}) and (\ref{2}), $%
F_{\alpha }$ is nowhere dense whenever $\overset{\circ }{F_{\alpha }^{\prime
}}=\emptyset .$ While the necessity is a little tricky but quickly proved,
the path to sufficiency is longer and more meandering.

Before laying out our strategy to prove sufficiency, we point out that the
above necessary and sufficient condition is satisfied (as it must be) if $f$
is usc or lsc, or when $X=\Bbb{R}^{N}.$ The latter yields a ``category''
variant of Crouzeix's theorem without the differentiability part, which
however can be given a straightforward direct proof (Remark \ref{rm2}).

A main ingredient that will come into play at the onset is the concept of
``topological'' essential extremum (Section \ref{topological}), which mimics
the familiar notion from measure theory. More specifically, if $\mathcal{T}$
denotes the topology of $X$ and $U\subset X$ is an open\footnote{%
The openness of $U$ is not important, but the definition will not be needed
when $U$ is not open.} subset, we define 
\begin{equation*}
\mathcal{T}\limfunc{ess}\sup_{U}f:=\inf \{\alpha \in \Bbb{R}:\{f_{|U}>\alpha
\}\text{ is of first category}\}
\end{equation*}
and, next, $\mathcal{T}\limfunc{ess}\inf_{U}f:=-\mathcal{T}\limfunc{ess}%
\sup_{U}(-f),$ which of course is also given by a similar formula.

These $\mathcal{T}$-essential extrema are discussed in the next section,
when $X$ is a topological space. Although the aim there is just the proof of
the elementary but basic Lemma \ref{lm1}, it seemed instructive to show that
they are fairly versatile since special cases not only include the pointwise
extrema, but also the classical essential extrema of measure theory, at
least in the case of the Lebesgue measure.

Evidently, such a simple-minded and natural definition must have come to
mind only moments after the discovery of the Baire categories\footnote{%
Figuratively speaking, since Baire category (1899) predates measure theory
(1901).}, but it appears to have remained rather confidential. A closely
related idea is explored at some length in two 1963 papers by Semadeni \cite
{Se63}, \cite{Se63b}, for a very specific and different purpose. Curiously,
this idea does not seem to resurface in other issues and we were unable to
find any subsequent record of its use.

In this paper, $\mathcal{T}$-essential extrema are pervasive because $%
\mathcal{T}\limfunc{ess}\inf_{X}f$ is a crucial value in the problem under
investigation, but their importance is primarily due to the combination of
two features. First, by design, they are unchanged after modification of the
function on a subset of first category. Second, quasiconvex functions have a
remarkable property relative to these extrema: If $f$ is quasiconvex (and $U$
is open), then $\mathcal{T}\limfunc{ess}\sup_{U}f=$ $\sup_{U}f,$ the \emph{\
pointwise} supremum of $f$ on $U,$ while $\mathcal{T}\limfunc{ess}%
\inf_{U}f=\inf_{U}f,$ the \emph{pointwise} infimum of $f$ on $U,$ provided
that equality holds when $U=X$ (Theorem \ref{th4}). These properties are
typical of lower and upper semicontinuous functions, respectively, but for
quasiconvex functions, they hold without any semicontinuity requirement.

Next, if $f$ and $g$ are quasiconvex functions that agree on a residual
subset of $X$ (equivalent functions) and if $x$ is a point of continuity of $%
g,$ Theorem \ref{th4} is instrumental in discovering and proving a necessary
and sufficient condition for $x$ to be a point of continuity of $f$ as well
(Theorem \ref{th6}). A corollary yields a necessary and sufficient condition
for the set of points of discontinuity of $f$ to be of first category if
this is true of the set of points of discontinuity of $g$ (Corollary \ref
{cor7}).

In turn, this can be used to find a sharp sufficient condition for the set
of points of discontinuity of a quasiconvex function $f$ to be of first
category, but now without having to know an equivalent quasiconvex function
with that property. Indeed, under suitable assumptions about $f,$ the usc
hull of $f$ is (quasiconvex and) real-valued and equivalent to $f.$ Since
the set of points of discontinuity of a semicontinuous function is of first
category, the desired condition follows from the sufficiency part of
Corollary \ref{cor7} (Theorem \ref{th9}). Ultimately, this leads to the
sufficiency of the conditions listed in Corollary \ref{cor10} and the main
goal of the paper is achieved. For a rather surprising by-product, see
Corollary \ref{cor11}.

In Section \ref{applications}, we give two simple sufficient (not necessary)
conditions for the set of points of discontinuity of a quasiconvex function
to be of first category. One of them (Theorem \ref{th12}) is used to show
that if the set of points of discontinuity of a quasiconvex function $f$ is
of first category, then with only one possible exception, $f^{-1}(\alpha )$
is either nowhere dense or has nonempty interior (Theorem \ref{th14}). Even
if $f$ is usc or lsc, the proof requires using Theorem \ref{th12} in the
nontrivial case when the function of interest is not semicontinuous. We also
show that $f$ and its lsc hull have the same points of continuity (Theorem 
\ref{th15}), which is false when the set of points of discontinuity of $f$
is of second category.

The last section is devoted to two complements of independent interest. We
prove that the sets $\overset{\circ }{F_{\alpha }}$ and $\overset{\circ }{
F_{\alpha }^{\prime }}$ (see (\ref{1}) and (\ref{2})), which play a key role
in the paper, are unchanged when the quasiconvex function $f$ is replaced by
an equivalent quasiconvex function (Theorem \ref{th17}). This clarifies the
connection between various results in the preceding sections.

Lastly, we prove that every usc quasiconvex function is quasicontinuous in
the (classical) sense of Kempisty. This was suggested by some technical
conditions related to the proof of Theorem \ref{th4} and seems to have so
far remained unnoticed, even when $X=\Bbb{R}^{N}.$

\section{Topological essential extrema\label{topological}}

As mentioned in the Introduction, if $X$ is a topological space with
topology $\mathcal{T},$ $f:X\rightarrow \overline{\Bbb{R}}$ is a given
function and $U\subset X$ is open and nonempty, we set 
\begin{multline}
\mathcal{T}\limfunc{ess}\sup_{U}f:=\inf \{\alpha \in \Bbb{R}:\{f_{|U}>\alpha
\}\text{ is of first category}\}=  \label{3} \\
\sup \{\alpha \in \Bbb{R}:\{f_{|U}>\alpha \}\text{ is of second category}\}.
\end{multline}
The second equality in (\ref{3}) follows from the fact that the collection
of upper level sets $\left( \{f_{|U}>\alpha \}\right) _{\alpha \in \Bbb{R}}$
is nonincreasing and, if convenient, the level sets $\{f_{|U}>\alpha \}$ can
be replaced by $\{f_{|U}\geq \alpha \}$ without prejudice. If $U=\emptyset ,$
we set $\mathcal{T}\limfunc{ess}\sup_{\emptyset }f:=-\infty .$ It is always
true that $\mathcal{T}\limfunc{ess}\sup_{U}f\leq \sup_{U}f$ and, if $X$ is a
Baire space and $f$ is continuous, then $\mathcal{T}\limfunc{ess}
\sup_{U}f=\sup_{U}f.$

When $\mathcal{T}$ is the discrete topology on a set $X,$ (\ref{3})
coincides with $\sup_{U}f$ since $\emptyset $ is the only subset of first
category. It is less obvious but shown in the example below that the
definition (\ref{3}) also incorporates the classical essential supremum of a
Lebesgue measurable function as a special case.

\begin{example}
\label{ex1}Let $X=\Bbb{R}^{N}$ and $\mathcal{T}=\mathcal{D},$ the density
topology, whose open subsets are $\emptyset $ and the Lebesgue measurable
subsets with density $1$ at each of their points (\cite{GoNeNi61}, \cite
{HaPa52}, \cite{LuMaZa86}). With this topology, $\Bbb{R}^{N}$ is not a
metric space ($\mathcal{D}$ is not first countable) and not even a tvs, but
it is a Baire space. If $f$ is measurable and $S\subset \Bbb{R}^{N}$ is a
measurable subset with density interior $U,$ then $\mathcal{D}\limfunc{ ess}%
\sup_{U}f=\limfunc{ess}\sup_{S}f$: First, $\limfunc{ ess}\sup_{S}f=\limfunc{
ess}\sup_{U}f$ because $U$ is the set of points at which $S$ has density $1,$
so that $S\backslash U$ is a null set (set of measure zero) by the Lebesgue
density theorem. Next, $\limfunc{ess}\sup_{U}f=\mathcal{D}\limfunc{ ess}%
\sup_{U}f$ because a subset is of first category for $\mathcal{\ D}$ if and
only if it is a null set. \newline
In addition, (\ref{3}) makes sense when $f$ is not measurable. For example,
if $A\subset X$ is not measurable and $f=\chi _{A},$ then $\mathcal{D}%
\limfunc{ess}\sup_{\Bbb{R}^{N}}\chi _{A}=1$ (if $A$ were of first category
for $\mathcal{D}$, it would be a null set and hence measurable). As odd as
it might seem, there is actually nothing new here: Although rarely if ever
used, the definition of $\limfunc{ess}\sup_{U}f$ makes sense for
non-measurable $f$ by using the Lebesgue outer measure $\mu _{N}^{*},$ viz., 
$\limfunc{ess}\sup_{U}f=\inf \{\alpha \in \Bbb{R}:\mu
_{N}^{*}(\{f_{|U}>\alpha \})=0\}.$ Since $\mu _{N}^{*}$ vanishes exactly on
null sets and the null sets are the subsets of first category for $\mathcal{D%
}$, this is just $\mathcal{D}\limfunc{ess}\sup_{U}f.$
\end{example}

Naturally, we can also define $\mathcal{T}\limfunc{ess}\inf_{U}f=-\mathcal{T}%
\limfunc{ess}\sup_{U}(-f),$ that is, 
\begin{multline}
\mathcal{T}\limfunc{ess}\inf_{U}f:=\sup \{\alpha \in \Bbb{R}:\{f_{|U}<\alpha
\}\text{ is of first category}\}=  \label{4} \\
\inf \{\alpha \in \Bbb{R}:\{f_{|U}<\alpha \}\text{ is of second category}\},
\end{multline}
when $U$ is open and nonempty and $\mathcal{T}\limfunc{ess}\inf_{\emptyset
}f:=\infty .$ Once again, the level sets $\{f_{|U}\leq \alpha \}$ can be
used instead of $\{f_{|U}<\alpha \}$ and $\inf_{U}f\leq \mathcal{T}\limfunc{
ess}\inf_{U}f,$ with equality (for instance) when $X$ is a Baire space and $%
f $ is continuous.

If $f$ and $g$ are two extended real valued functions on $X,$ it is natural
to define the $\mathcal{T}$-equivalence of $f$ and $g,f\sim _{\mathcal{T}}g$
for short, by 
\begin{equation}
f\sim _{\mathcal{T}}g\text{ if }f=g\text{ on a residual subset of }X,
\label{5}
\end{equation}
where, as usual, a residual subset is the complement of a set of first
category. Thus, $f\sim _{\mathcal{T}}g$ if and only if $f\neq g$ on a set of
first category. Since the union of two sets of first category is of first
category, 
\begin{equation}
f\sim _{\mathcal{T}}g\text{ }\Rightarrow \mathcal{T}\limfunc{ess}\sup_{U}f=%
\mathcal{T}\limfunc{ess}\sup_{U}g\text{ and }\mathcal{T}\limfunc{ess}
\inf_{U}f=\mathcal{T}\limfunc{ess}\inf_{U}g,  \label{6}
\end{equation}
for every open subset $U\subset X.$

The question whether the pointwise and $\mathcal{T}$-essential extrema of a
real-valued function coincide on every open subset will be of central
importance in the next section. In turn, this will hinge on the following
simple lemma.

\begin{lemma}
\label{lm1}Let $X$ be a topological space with topology $\mathcal{T}.$ For
every function $f:X\rightarrow \Bbb{R},$ the following statements are
equivalent.\newline
(i) $\sup_{U}f=\mathcal{T}\limfunc{ess}\sup_{U}f$ for every open subset $%
U\subset X.$ \newline
(ii) For every $x_{0}\in X,$ every open subset $U\subset X$ containing $%
x_{0} $ and every $\varepsilon >0,$ the set $\{x\in
U:f(x)>f(x_{0})-\varepsilon \}$ is of second category.\newline
Likewise, the following statements are equivalent: \newline
(i') $\inf_{U}f=\mathcal{T}\limfunc{ess}\inf_{U}f$ for every open subset $%
U\subset X.$ \newline
(ii') For every $x_{0}\in X,$ every open subset $U\subset X$ containing $%
x_{0}$ and every $\varepsilon >0,$ the set $\{x\in
U:f(x)<f(x_{0})+\varepsilon \}$ is of second category.
\end{lemma}

\begin{proof}
(i) $\Rightarrow $ (ii). Suppose that (i) holds and, by contradiction,
assume that there are $x_{0}\in X,$ an open subset $U\subset X$ containing $%
x_{0}$ and some $\varepsilon >0$ such that $\{x\in
U:f(x)>f(x_{0})-\varepsilon \}$ is of first category. Then, $\mathcal{T}%
\limfunc{ess}\sup_{U}f\leq f(x_{0})-\varepsilon <f(x_{0})\leq \sup_{U}f,$
which contradicts (i).

(ii) $\Rightarrow $ (i). Let $U\subset X$ be an open subset. We argue by
contradiction, thereby assuming that $\sup_{U}f>\mathcal{T}\limfunc{ ess}%
\sup_{U}f.$ If so, $U$ is not empty (otherwise, both suprema are $-\infty $
) and $\mathcal{T}\limfunc{ess}\sup_{U}f<\infty .$ Thus, the assumption $%
\sup_{U}f>\mathcal{T}\limfunc{ess}\sup_{U}f$ implies the existence of $%
x_{0}\in X$ such that $\mathcal{T}\limfunc{ess}\sup_{U}f<f(x_{0})\leq
\sup_{U}f.$ Choose $\varepsilon >0$ small enough that $\mathcal{T}\limfunc{%
ess}\sup_{U}f<f(x_{0})-\varepsilon .$ By (ii), $\{x\in
U:f(x)>f(x_{0})-\varepsilon \}$ is of second category, so that $\mathcal{T}%
\limfunc{ess}\sup_{U}f\geq f(x_{0})-\varepsilon ,$ which is a contradiction.

That (i') $\Leftrightarrow $ (ii') follows by replacing $f$ by $-f$ above.
\end{proof}

\section{Extrema of quasiconvex functions on open subsets\label{extrema}}

From this point on, $X$ is a Baire tvs with topology $\mathcal{T}.$ We shall
show that under a simple necessary and sufficient condition on the
quasiconvex function $f,$ the $\mathcal{T}$-essential extrema of $f$ on any
open subset $U\subset X$ coincide with its corresponding pointwise extrema
on $U.$ This will follow from Lemma \ref{lm1} after we expose some
properties of convex sets relative to Baire category (Lemma \ref{lm3} below).

Convex subsets of infinite dimensional spaces have less pleasant features
than their finite dimensional counterparts. However, an important one
remains unchanged, which is spelled out in the following remark.

\begin{remark}
\label{rm1}If $C\subset X$ is a convex subset with $\overset{\circ }{C}\neq %
\emptyset ,$ then $\overset{\circ }{C}$ is convex, the closures of $C$ and $%
\overset{\circ }{C}$ are the same (i.e., $\overline{C}$) and the interior of 
$C$ and $\overline{C}$ are the same (i.e., $\overset{\circ }{C}$) 
\cite[p.105]{Be74}. In particular, $\partial C=\partial \overset{\circ }{C}=%
\partial \overline{C}.$ When $\dim X=\infty ,$ all these properties break
down when $\overset{\circ }{C}=\emptyset .$ However, $\overline{C}$ is
always convex \cite[p. 103]{Be74}.
\end{remark}

The next lemma will be used several times, including in the proof of part
(ii) of Lemma \ref{lm3} below.

\begin{lemma}
\label{lm2}Let $U$ and $C$ be subsets of $X$ with $U$ open and $C$ convex.
If $A\subset X$ is of first category and $U\backslash A\subset C,$ then $%
U\subset C.$
\end{lemma}

\begin{proof}
It suffices to show that if $x_{0}\in U\cap A\neq \emptyset ,$ then $%
x_{0}\in C.$ After translation, it is not restrictive to assume that $%
x_{0}=0.$

The set $V:=U\cap (-U)\subset U$ is an open neighborhood of $0$ and $A\cup
(-A)$ is of first category. Thus, $V\backslash (A\cup (-A))\neq \emptyset .$
If $x\in V\backslash (A\cup (-A)),$ both $x$ and $-x$ are in $V\backslash
A\subset U\backslash A,$ hence in $C.$ Since $C$ is convex, $0=\frac{1}{2}x+%
\frac{1}{2}(-x)\in C.$
\end{proof}

\begin{lemma}
\label{lm3}Let $C\subset X$ be a convex subset and $U\subset X$ be an open
subset.\newline
(i) If $C$ is of second category and $U\cap C\neq \emptyset ,$ then $U\cap C$
is of second category.\newline
(ii) If $U\cap (X\backslash C)\neq \emptyset ,$ then $U\cap (X\backslash C)$
is of second category.
\end{lemma}

\begin{proof}
(i) By contradiction, assume that $U\cap C$ is of first category. Since $%
U\cap C\neq \emptyset ,$ pick $x_{0}\in U\cap C.$ After translation, it is
not restrictive to assume $x_{0}=0\in U\cap C.$ Let $x\in C$ be given. Then, 
$\frac{1}{n}x\in C$ for every $n\in \Bbb{N}$ and $\frac{1}{n}x\in U$ if $n$
is large enough since $U$ is open and the scalar multiplication is
continuous. Thus, $\frac{1}{n}x\in U\cap C$ for every $x\in C$ and some $%
n\in \Bbb{N}.$ This means that $C\subset \cup _{n\in \Bbb{N}}n(U\cap C).$
Since $U\cap C$ is of first category, the same thing is true of $n(U\cap C)$
for every $n,$ so that $\cup _{n\in \Bbb{N}}n(U\cap C)$ is of first
category. But then, $C$ is of first category, which is a contradiction.

(ii) Since $U\cap (X\backslash C)=U\backslash C\neq \emptyset,$ it is
obvious that $U\nsubseteq C.$ By contradiction, assume that $U\backslash C$
is of first category and note that $U\backslash (U\backslash C)=U\cap
C\subset C.$ Since $U$ is open, $C$ is convex and $U\backslash C$ is of
first category, it follows from Lemma \ref{lm2} that $U\subset C,$ which is
a contradiction.
\end{proof}

Given a function $f:X\rightarrow \Bbb{R}$ and $\alpha \in \Bbb{R},$ recall
the notation $F_{\alpha }:=\{x\in X:f(x)<\alpha \}$ and $F_{\alpha }^{\prime
}:=\{x\in X:f(x)\leq \alpha \}$ in (\ref{1}) and (\ref{2}). This notation
will only be used without further specification when the function of
interest is called $f.$

\begin{theorem}
\label{th4}Let $f:X\rightarrow \overline{\Bbb{R}}$ be quasiconvex. \newline
(i) $\sup_{U}f=\mathcal{T}\limfunc{ess}\sup_{U}f$ for every open subset $%
U\subset \Bbb{R}^{N}.$\newline
(ii) $\inf_{U}f=\mathcal{T}\limfunc{ess}\inf_{U}f$ for every open subset $%
U\subset X$ if and only if this is true when $U=X.$
\end{theorem}

\begin{proof}
Since changing $f$ into $\arctan f$ does not affect quasiconvexity and since
it is readily checked that $\arctan $ commutes with $\mathcal{T}\limfunc{ess}%
\sup_{U}$ and $\mathcal{T}\limfunc{ess}\inf_{U}$ (as well as with $\sup_{U}$
and $\inf_{U}$), we assume with no loss of generality that $f$ is finite.

(i) We show that the condition (ii) of Lemma \ref{lm1} holds and use its
equivalence with part (i) of that lemma.

Pick $x_{0}\in X$ along with an open subset $U\subset X$ containing $x_{0}$
and $\varepsilon >0.$ The set $\{x\in U:f(x)>f(x_{0})-\varepsilon \}$
contains $x_{0}$ and is the (nonempty) intersection $U\cap ($ $X\backslash
C) $ where $C:=F_{f(x_{0})-\varepsilon }^{\prime }$ is convex. That $U\cap ($
$X\backslash C)$ is of second category follows from part (ii) of Lemma \ref
{lm3}.

(ii) It is obvious that $\inf_{X}f=\mathcal{T}\limfunc{ess}\inf_{X}f$ is
necessary. Conversely, assuming this, we show that condition (ii') of Lemma 
\ref{lm1} holds and use its equivalence with part (i') of that lemma.

Pick $x_{0}\in X$ along with an open subset $U\subset X$ containing $x_{0}$
and $\varepsilon >0.$ The set $\{x\in U:f(x)<f(x_{0})+\varepsilon \}$
contains $x_{0}$ and is the (nonempty) intersection $U\cap C$ where $%
C:=F_{f(x_{0})+\varepsilon }$ is convex. It is also of second category by
definition of $\mathcal{T}\limfunc{ ess}\inf_{X}f$ since $\mathcal{T}%
\limfunc{ess}\inf_{X}f=\inf_{X}f\leq f(x_{0})<f(x_{0})+\varepsilon .$ That $%
U\cap C$ is of second category follows from part (i) of Lemma \ref{lm3}.
\end{proof}

\section{Points of continuity of equivalent quasiconvex functions\label%
{equivalent}}

Suppose that $f,g:X\rightarrow \Bbb{R}$ are quasiconvex functions and that $%
f\sim _{\mathcal{T}}g$ (see (\ref{5})). In this section, we give a complete
answer to the following question: If $x$ is a point of continuity of one
function, when is $x$ a point of continuity of the other? That such a
question can be answered is due to the properties highlighted in Theorem \ref
{th4}, combined with the quasiconvexity of both functions.

We confine attention to real-valued functions, because defining points of
continuity is possible, but unnecessarily convoluted, for extended
real-valued functions. The issue is that since the \emph{size} of such sets
is the matter of interest, it would not be adequate to require, as is
normally done, that the function be finite at a point of continuity.
Instead, it is much simpler to use the usual $\arctan $ trick to reduce the
problem to the real-valued case.

We shall implicitly use the elementary remark that if $f$ is any real-valued
function on $X,$ then $\mathcal{T}\limfunc{ess}\inf_{X}f<\infty ,$ for
otherwise $F_{n}$ is of first category for every $n\in \Bbb{N}$ and so $%
X=\cup _{n\in \Bbb{N}}F_{n}$ is of first category, which is absurd since $X$
is a Baire space. As a result, we will never have to be concerned that any $%
\mathcal{T}$-essential infimum might be $\infty .$ They can, of course, be $-%
\infty .$ We begin with a special case.

\begin{lemma}
\label{lm5}Let $f,g:X\rightarrow \Bbb{R}$ be quasiconvex functions such that 
$f\sim _{\mathcal{T}}g,$ so that $\mathcal{T}\limfunc{ess}\inf_{X}f=\mathcal{%
T}\limfunc{ess}\inf_{X}g:=m\geq -\infty $ (see (\ref{6})). \newline
(i) If also $\inf_{X}f=m=\inf_{X}g$ (always true if $m=-\infty $), then $f$
and $g$ have the same points of continuity and achieve a common value at
such points. \newline
(ii) $\max \{f,m\}$ and $\max \{g,m\}$ have the same points of continuity
and achieve a common value at such points.
\end{lemma}

\begin{proof}
(i) Let $x$ denote a point of continuity of $f,$ so that for every $%
\varepsilon >0,$ there is an open neighborhood $U$ of $x$ such that $%
f(U)\subset [f(x)-\varepsilon ,f(x)+\varepsilon ].$ Thus, $\inf_{U}f\geq
f(x)-\varepsilon $ and $\sup_{U}f\leq f(x)+\varepsilon .$ By Theorem \ref
{th4}, this is the same as $\mathcal{T}\limfunc{ess}\inf_{U}f\geq
f(x)-\varepsilon $ and $\mathcal{T}\limfunc{ess}\sup_{U}f\leq
f(x)+\varepsilon .$

Since $f=g$ a.e., the $\mathcal{T}$-essential extremum is unchanged when $f$
is replaced by $g,$ so that $\mathcal{T}\limfunc{ess}\inf_{U}g\geq
f(x)-\varepsilon $ and $\mathcal{T}\limfunc{ ess}\sup_{U}g\leq
f(x)+\varepsilon .$ By using once again Theorem \ref{th4}, it follows that $%
\inf_{U}g\geq f(x)-\varepsilon $ and $\sup_{U}g\leq f(x)+\varepsilon ,$
whence $g(U)\subset [f(x)-\varepsilon ,f(x)+\varepsilon ].$ In particular, $%
g(x)\in [f(x)-\varepsilon ,f(x)+\varepsilon ].$ Since $\varepsilon >0$ is
arbitrary, it follows that $g(x)=f(x)$ and hence that $g(U)\subset
[g(x)-\varepsilon ,g(x)+\varepsilon ],$ which proves the continuity of $g$
at $x.$

In summary, the points of continuity of $f$ are points of continuity of $g$
and $g=f$ at such points. By exchanging the roles of $f$ and $g,$ the
converse is true.

(ii) Just use (i) with $\max \{f,m\}$ and $\max \{g,m\},$ respectively.
Quasiconvexity is not affected and $\max \{f,m\}\sim _{\mathcal{T}}\max
\{g,m\}$ while $\mathcal{T}\limfunc{ess}\inf_{X}\max \{f,m\}=\mathcal{T}%
\limfunc{ess}\inf_{X}\max \{g,m\}=m,$ so that $\inf_{X}\max
\{f,m\}=m=\inf_{X}\max \{g,m\}$ follows from $\max \{f,m\}\geq m,\max
\{g,m\}\geq m$ and $\inf_{X}\leq \mathcal{T}\limfunc{ess}\inf_{X}.$
\end{proof}

The next example shows that the condition $\inf_{X}f=m=\inf_{X}g$ cannot be
dropped in part (i) of Lemma \ref{lm5}.

\begin{example}
\label{ex2}Let $X$ be infinite dimensional and let $H\subset X\;$be a dense
hyperplane of first category. It was first shown by Arias de Reyna \cite
{Ar80} in 1980 that such a hyperplane exists when $X$ is a separable Banach
space, but they also exist in general (Saxon \cite{Sa92}; all known proofs,
including \cite{Ar80}, assume a weak form of the Continuum Hypothesis). If $%
f=\chi _{X\backslash H}$ and $g=1,$ then both $f$ and $g$ are quasiconvex
and $f\sim _{\mathcal{T}}g.$ However, $\inf_{X}f=0\neq m=1=\inf_{X}g$ and
indeed $f$ is nowhere continuous while $g$ is everywhere continuous.
\end{example}

\begin{theorem}
\label{th6}Let $f,g:X\rightarrow \Bbb{R}$ be quasiconvex functions such that 
$f\sim _{\mathcal{T}}g,$ so that $\mathcal{T}\limfunc{ess}\inf_{X}f=\mathcal{%
T}\limfunc{ ess}\inf_{X}g:=m\geq -\infty .$ If $x$ is a point of continuity
of $g,$ then $g(x)\geq m.$ Furthermore, a point of continuity $x$ of $g$ is
not one of $f$ if and only if $m>-\infty ,g(x)=m$ and $x\in F_{m}^{\prime
}\cap \left( \cup _{\alpha <m}\overline{F}_{\alpha }\right) .$
\end{theorem}

\begin{proof}
If $m=-\infty ,$ the only nontrivial property follows from part (i) of Lemma 
\ref{lm5}. Below, $m>-\infty $ and $x$ denotes a point of continuity of $g.$
To make the proof more transparent, we give four preliminary results.

(i) $g(x)\geq m.$ By contradiction, assume that $g(x)<m$ and let $\alpha \in 
\Bbb{R}$ be such that $g(x)<\alpha <m.$ By definition of $m,$ the set $%
G_{\alpha }:=\{y\in X:g(y)<\alpha \}$ is of first category. On the other
hand, $G_{\alpha }=g^{-1}((-\infty ,\alpha ))$ is a neighborhood of $x$ in $%
X $ since $g$ is continuous at $x.$ Therefore, $G_{\alpha }$ is of second
category, which is a contradiction.

(ii) $x\notin \cup _{\alpha <m}\overline{G}_{\alpha }.$ By (i), $g(x)\geq m.$
By contradiction, assume that $x\in \overline{G}_{\alpha }$ for some $\alpha
<m,$ whence $U\cap G_{\alpha }\neq \emptyset $ for every neighborhood $U$ of 
$x.$ Choose $\varepsilon >0$ such that $\alpha <g(x)-\varepsilon .$ Since $g$
is continuous at $x,$ then $U:=g^{-1}((g(x)-\varepsilon ,\infty ))$ is a
neighborhood of $x$ but $U\cap G_{\alpha }=\emptyset $ and a contradiction
is reached.

(iii) If $g(x)>m$ \emph{or} $f(x)>m,$ then $f$ is continuous at $x.$ Indeed, 
$x$ is a point of continuity of $\max \{g,m\}.$ Thus, by part (ii) of Lemma 
\ref{lm5}, it is a point of continuity of $\max \{f,m\}$ and $\max
\{g(x),m\}=\max \{f(x),m\}.$ As a result, $g(x)>m$ \emph{\ and} $f(x)>m.$
Let $\varepsilon >0$ be small enough that $m<f(x)-\varepsilon $ and let $%
I_{\varepsilon }:=(f(x)-\varepsilon ,f(x)+\varepsilon ).$ Then, $(\max
\{f,m\})^{-1}(I_{\varepsilon })$ is a neighborhood $W_{\varepsilon }$ of $x.$
From the choice of $\varepsilon ,$ it is obvious that $W_{\varepsilon
}=f^{-1}(I_{\varepsilon }).$ Since this is true for every $\varepsilon >0$
small enough, it follows that $f$ is continuous at $x.$

(iv) If $f(x)=m$ and $x\notin \cup _{\alpha <m}\overline{F}_{\alpha },$ then 
$f$ is continuous at $x.$ Recall that $x$ is a point of continuity of $\max
\{f,m\}$ (see (iii)). Accordingly, given $\varepsilon >0,$ the set $V:=(\max
\{f,m\})^{-1}((-\infty ,m+\varepsilon ))$ is a neighborhood of $x$ and, if $%
y\in V,$ then $f(y)<m+\varepsilon .$ On the other hand, $x\notin \overline{F}%
_{m-\frac{\varepsilon }{2}},$ so that there is a neighborhood $W$ of $x$
such that $W\cap F_{m-\frac{\varepsilon }{2}}=\emptyset .$ Equivalently, $%
f(y)\geq m-\frac{\varepsilon }{2}$ for $y\in W$ and so, if $y\in V\cap W,$ a
neighborhood of $x,$ then $m-\varepsilon <f(y)<m+\varepsilon .$ Since $%
f(x)=m,$ this reads $|f(y)-f(x)|<\varepsilon ,$ which shows that $f$ is
continuous at $x.$

We now prove the theorem. That $g(x)\geq m$ is (i). Next, if $x$ is not a
point of continuity of $f,$ then $g(x)=m$ and $f(x)\leq m,$ i.e., $x\in
F_{m}^{\prime },$ by (i) and (iii). That $x\in \cup _{\alpha <m}\overline{F}%
_{\alpha }$ is obvious if $x\in F_{m}=\cup _{\alpha <m}F_{\alpha }$ and so,
since $x\in F_{m}^{\prime }$ is already known, it suffices to show that $%
x\in \cup _{\alpha <m}\overline{F}_{\alpha }$ if $f(x)=m.$ This follows from
(iv) since $x$ is not a point of continuity of $f.$ This proves the
necessity of the conditions. By (ii) for $f$ instead of $g$ (i.e., if $x\in
\cup _{\alpha <m}\overline{F}_{\alpha },$ then $f$ is not continuous at $x$%
), their sufficiency is obvious.
\end{proof}

With $f$ and $g$ of Example \ref{ex2}, $g=m=1$ is constant but $\overline{F}%
_{\alpha }=X$ for every $\alpha \in (0,1)$ and $F_{1}^{\prime }=X$ so that,
by Theorem \ref{th6}, $f$ is nowhere continuous, which is indeed true. If
the hyperplane $H$ in that example is of second category (such hyperplanes
are easily constructed in any infinite dimensional tvs \cite[p. 14]{Va82}),
then Theorem \ref{th6} is not applicable since $f\neq g$ on $H,$ of second
category.

\begin{corollary}
\label{cor7}Let $f,g:X\rightarrow \Bbb{R}$ be quasiconvex with $f\sim _{%
\mathcal{T}}g,$ so that $\mathcal{T}\limfunc{ess}\inf_{X}f=\mathcal{T}%
\limfunc{ess}\inf_{X}g:=m$ $\geq -\infty .$ Suppose also that the set of
points of discontinuity of $g$ is of first category. Then, the set of points
of discontinuity of $f$ is of first category if and only if one of the
following two conditions holds:\newline
(i) $m=-\infty $.\newline
(ii) $m>-\infty $ and $f^{-1}(m)\cap \overline{F}_{\alpha }$ is of first
category for every $\alpha <m$.
\end{corollary}

\begin{proof}
If $x$ is a point of discontinuity of $f,$ it is either a point of
discontinuity of both $f$ and $g,$ or a point of continuity of $g$ which is
not one of $f.$ Since the set $B$ of points of discontinuity of $g$ is of
first category, the same thing is true of the set of points of discontinuity
of $f$ if and only if the points of $X\backslash B$ (i.e., points of
continuity of $g$) that are not points of continuity of $f$ belong in a set
of first category. By Theorem \ref{th6}, this happens if and only if (1) $%
m=-\infty $ (and then every point of $X\backslash B$ is a point of
continuity of $f$) or (2) $m>-\infty $ and $\{x\in F_{m}^{\prime }\cap
\left( \cup _{\alpha <m}\overline{F}_{\alpha }\right) :g(x)=m,x\in
X\backslash B\}$ is of first category. For clarity, rewrite this set as $%
(g^{-1}(m)\backslash B)\cap F$ where $F:=F_{m}^{\prime }\cap \left( \cup
_{\alpha <m}\overline{F}_{\alpha }\right) .$ Since $B$ is of first category
and $(g^{-1}(m)\backslash B)\cap F=((g^{-1}(m)\cap F)\backslash B,$ the sets 
$(g^{-1}(m)\backslash B)\cap F$ and $g^{-1}(m)\cap F$ are of first category
simultaneously. Furthermore, since $f\sim _{\mathcal{T}}g,$ the sets $%
g^{-1}(m)\cap F$ and $f^{-1}(m)\cap F$ are also of first category
simultaneously. But $f^{-1}(m)\subset F_{m}^{\prime },$ so that $%
f^{-1}(m)\cap F$ is just $f^{-1}(m)\cap \left( \cup _{\alpha <m}\overline{F}%
_{\alpha }\right) .$ Since the union is actually a countable union, this set
is of first category if and only if $f^{-1}(m)\cap \overline{F}_{\alpha }$
is of first category for every $\alpha <m.$
\end{proof}

In Example \ref{ex2}, $m=1,\overline{F}_{\alpha }=X$ if $\alpha \in (0,1)$
and $f^{-1}(1)$ is residual (hence of second category since $X$ is a Baire
space) so that condition (ii) of Corollary \ref{cor7} fails.

\section{Continuity at the points of a residual subset\label{continuity}}

We shall give a few equivalent necessary and sufficient conditions for the
set of points of discontinuity of a quasiconvex function $f:X\rightarrow 
\Bbb{R}$ to be of first category (Corollary \ref{cor10}). We begin with the
construction of an auxiliary function.

\begin{lemma}
\label{lm8}Let $f:X\rightarrow \Bbb{R}$ be quasiconvex and set 
\begin{equation}
\overline{f}(x):=\inf \{\alpha \in \Bbb{R}:x\in \overset{\circ }{F}_{\alpha
}\}.  \label{7}
\end{equation}
The following properties hold:\newline
(i) $f\leq \overline{f}.$\newline
(ii) $\overline{f}$ is quasiconvex and upper semicontinuous (usc).\newline
(iii) If, in addition, $\overset{\circ }{F_{\alpha }}\neq \emptyset $
whenever $\alpha >m:=\mathcal{T}\limfunc{ess}\inf_{X}f,$ then $\overline{f}$
is real-valued and $f\sim _{\mathcal{T}}\overline{f}.$
\end{lemma}

\begin{proof}
(i) By contradiction, if $x\in X$ and $\overline{f}(x)<f(x),$ there is $%
\alpha <f(x)$ such that $x\in \overset{\circ }{F_{\alpha }}\subset F_{\alpha
}$ and so $f(x)<\alpha <f(x),$ which is absurd.

(ii) Both properties follow from the simple remark that, if $\beta \in \Bbb{R%
},$ then $\{x\in X:\overline{f}(x)<\beta \}=\cup _{\alpha <\beta }\overset{%
\circ }{F_{\alpha }},$ an open convex subset since the sets $\overset{\circ 
}{F_{\alpha }}$ are convex and linearly ordered.

(iii) By (i), $\overline{f}(x)>-\infty $ since $f$ is real-valued. It
remains to show that if it is assumed that $\overset{\circ }{F_{\alpha }}
\neq \emptyset $ when $\alpha >m,$ then $\overline{f}(x)<\infty $ and $f\sim
_{\mathcal{T}}\overline{f}.$ From now on, we set $\Bbb{Q}_{m}:=\Bbb{Q}
\backslash \{m\},$ a countable dense subset of $\Bbb{R}.$ Of course, $\Bbb{Q}
_{m}=\Bbb{Q}$ if $m=-\infty $ or $m$ is irrational.

By definition of $\overline{f}$ in (\ref{7}), $\overline{f}(x)<\infty $ for
every $x\in X$ if (and in fact only if) $C:=\cup _{\alpha \in \Bbb{Q}_{m}}%
\overset{\circ }{F_{\alpha }}=X.$ By contradiction, assume $C\neq X.$ Since $%
C$ is convex, it follows from part (ii) of Lemma \ref{lm3} with $U=X$ that $%
X\backslash C$ is of second category. Now, if $x\in X\backslash C,$ there is 
$\alpha \in \Bbb{Q}_{m}$ such that $\alpha >\max \{m,f(x)\}$ (recall that $%
m=\infty $ does not occur), so that $x\in F_{\alpha }.$ Actually, $x\in
\partial F_{\alpha }$ since $x\notin \overset{\circ }{F_{\alpha }}\subset C$
and $\partial F_{\alpha }=\partial \overset{\circ }{F_{\alpha }}$ by Remark 
\ref{rm1} and the standing assumption that $\overset{\circ }{F_{\alpha }}%
\neq \emptyset $ when $\alpha >m.$ Thus, $\partial F_{\alpha }$ is nowhere
dense (recall that boundaries of open subsets are nowhere dense). As a
result, $X\backslash C\subset \cup _{\alpha \in \Bbb{Q}_{m},\alpha
>m}\partial F_{\alpha },$ a set of first category since $\Bbb{Q}_{m}$ is
countable. Thus, $X\backslash C$ is of first category. This contradiction
shows that $C=X.$

To complete the proof, we show that $f\sim _{\mathcal{T}}\overline{f}.$ Set $%
\Sigma :=\{x\in X:f(x)\neq \overline{f}(x)\}=\{x\in X:f(x)<\overline{f}(x)\}$
by part (i). By the denseness of $\Bbb{Q}_{m}$ in $\Bbb{R},$ $\Sigma =\cup
_{\alpha \in \Bbb{Q}_{m}}\Sigma (\alpha )$ where $\Sigma (\alpha ):=\{x\in
X:f(x)<\alpha <\overline{f}(x)\}.$ Since $\Bbb{Q}_{m}$ is countable, it
suffices to show that $\Sigma (\alpha )$ is of first category for every $%
\alpha \in \Bbb{Q}_{m}.$

If $x\in \Sigma (\alpha ),$ then $x\in F_{\alpha }$ and, if $\alpha <m,$ it
follows from the definition of $m$ that $F_{\alpha }$ is of first category.
If now $\alpha >m,$ then $\overset{\circ }{F_{\alpha }}\neq \emptyset $ is
assumed, so that once again $\partial F_{\alpha }=\partial \overset{\circ }{
F_{\alpha }}$ and $\partial F_{\alpha }$ is nowhere dense. On the other
hand, it is trivial that $x\in \partial F_{\alpha },$ for if $x\in \overset{
\circ }{F_{\alpha }},$ it follows from (\ref{7}) that $\overline{f}(x)\leq
\alpha ,$ which is a contradiction.

Since $m\notin \Bbb{Q}_{m},$ the above shows that $\Sigma (\alpha )$ is
always contained in a subset of $X$ of first category (i.e., $F_{\alpha }$
or $\partial F_{\alpha }$), so that it is of first category.
\end{proof}

It is not hard to check that $\overline{f}$ in (\ref{7}) is the usc hull of $%
f$ (smallest usc function $g\geq f$), but this not relevant. In general, $%
f\sim _{\mathcal{T}}\overline{f}$ is false: If $f$ is a discontinuous linear
form on $X,$ then $\overset{\circ }{F_{\alpha }}=\emptyset $ for every $%
\alpha ,$ so that $\overline{f}=\infty \neq f$ everywhere.

Under the assumption made in part (iii) of Lemma \ref{lm8}, $\overline{f}$
is real-valued and usc, so that its set of points of discontinuity is of
first category. This is often quoted only when $X$ is a complete metric
space, but true and elementary in every topological space; see for instance 
\cite[Lemma 2.1]{Na72}. Even though $f\sim _{\mathcal{T}}\overline{f},$ the
set of points of discontinuity of $f$ need not be of first category: In
Example \ref{ex2}, $\overline{f}=g$ but $f$ is nowhere continuous. However,
Corollary \ref{cor7} shows at once which condition is missing:

\begin{theorem}
\label{th9}Let $f:X\rightarrow \Bbb{R}$ be quasiconvex and set $m:=\mathcal{T%
}\limfunc{ess}\inf_{X}f.$ Suppose also that (i) $\overset{\circ }{F_{\alpha }%
}\neq \emptyset $ if $\alpha >m$ and that (ii) $f^{-1}(m)\cap \overline{F}%
_{\alpha }$ is of first category if $m>-\infty $ and $\alpha <m.$ Then, the
set of points of discontinuity of $f$ is of first category.
\end{theorem}

\begin{proof}
Condition (i) is the same as in part (iii) of Lemma \ref{lm8}, so that, as
noted before the theorem, the set of points of discontinuity of $\overline{f}
$ is of first category. Since also $f\sim _{\mathcal{T}}\overline{f}$ and $%
\overline{f}$ is real-valued and quasiconvex (Lemma \ref{lm8}), the set of
points of discontinuity of $f$ is of first category by Corollary \ref{cor7}
with $g=\overline{f}.$
\end{proof}

When $X=\Bbb{R}^{N},$ ``first category'', ``nowhere dense'' and ``empty
interior'' are synonymous for convex subsets. As a result, the hypotheses of
Theorem \ref{th9} always hold since $F_{\alpha }$ is of first (second)
category when $\alpha <m$ ($\alpha >m)$ by definition of $m.$ Thus, the set
of points of discontinuity of a real-valued quasiconvex function on $\Bbb{R}%
^{N}$ is always of first category.

\begin{remark}
\label{rm2}Since the set of points of discontinuity of an lsc function is of
first category, Theorem \ref{th9} when $X=\Bbb{R}^{N}$ follows at once from
the remark that the points of discontinuity of $f$ are exactly those of its
lsc hull\footnote{%
The $\arctan $ trick can be used to make sure that $\underline{f}$ is finite.%
}$\underline{f}$ (largest lsc function $g\leq f$). This is shown in 
\cite[Proposition 3.5]{Cr05} and is false if $\dim X=\infty $ (but see
Theorem \ref{th15}). Since $\underline{f}$ is quasiconvex, this also implies
that the $\sigma $-porosity result in \cite[Theorem 19]{BoWa05} in the lsc
case actually applies equally to $\underline{f}$ and $f$ and so does not
require lower semicontinuity.
\end{remark}

Another quick proof of Theorem \ref{th9} when $X=\Bbb{R}^{N}$ is that the
set of points of discontinuity of any real-valued function $f$ being an $%
\mathcal{F}_{\sigma },$ it can only be of first category or have nonempty
interior. Since it has measure zero when $f$ is quasiconvex (\cite{ChCr87}, 
\cite{Cr81}), the latter is impossible. Our attempts to use the $\mathcal{F}%
_{\sigma }$ property to get a shorter proof of Theorem \ref{th9} were not
successful when $\dim X=\infty .$ However, it is noteworthy that it implies
that if the set of points of discontinuity of a real-valued function $f$ is
not of first category, then $f$ is discontinuous at \emph{every} point of a
nonempty open subset.

As a corollary, we obtain necessary and sufficient conditions for the set of
points of discontinuity of $f$ to be of first category.

\begin{corollary}
\label{cor10}Let $f:X\rightarrow \Bbb{R}$ be quasiconvex and let $m:=%
\mathcal{T}\limfunc{ess}\inf_{X}f.$ The following statements are equivalent:%
\newline
(i) The set of points of discontinuity of $f$ is of first category.\newline
(ii) $F_{\alpha }$ is nowhere dense whenever $\overset{\circ }{F_{\alpha
}^{\prime }}=\emptyset .$\newline
(iii) $F_{\alpha }$ is nowhere dense whenever $\overset{\circ }{F_{\alpha
}^{\prime }}=\emptyset ,\alpha \neq m.$ \newline
(iv) $\overset{\circ }{F_{\alpha }^{\prime }}\neq \emptyset $ if $\alpha
>m,F_{\alpha }$ is nowhere dense if $m>-\infty $ and $\alpha <m.$\newline
(v) $\overset{\circ }{F_{\alpha }}\neq \emptyset $ if $\alpha >m,F_{\alpha }$
is nowhere dense if $m>-\infty $ and $\alpha <m.$ \newline
(vi) $\overset{\circ }{F_{\alpha }}\neq \emptyset $ if $\alpha
>m,f^{-1}(m)\cap $ $\overline{F}_{\alpha }$ is of first category if $%
m>-\infty $ and $\alpha <m.$
\end{corollary}

\begin{proof}
That (vi) $\Rightarrow $ (i) follows from Theorem \ref{th9}. Thus, it
remains to show that (i) $\Rightarrow $ (ii) $\Rightarrow $ (iii) $%
\Rightarrow $ (iv) $\Rightarrow $ (v) $\Rightarrow $ (vi).

(i) $\Rightarrow $ (ii). Assume that the set $A$ of points of discontinuity
of $f$ is of first category. Observe that if $f$ is continuous at $x\in 
\overline{F}_{\alpha },$ then $x\in F_{\alpha }^{\prime }.$

Now, let $\alpha \in \Bbb{R}$ be such that $F_{\alpha }$ is \emph{not}
nowhere dense, so that $\overline{F}_{\alpha }$ has nonempty interior $U.$
If $x\in U\backslash A,$ then $f$ is continuous at $x$ and so, from the
above remark, $U\backslash A\subset F_{\alpha }^{\prime }.$ Since $F_{\alpha
}^{\prime }$ is convex, it follows from Lemma \ref{lm2} that $U\subset
F_{\alpha }^{\prime }$ and hence that $\overset{\circ }{F_{\alpha }^{\prime }%
}\neq \emptyset .$

Thus, $\overset{\circ }{F_{\alpha }^{\prime }}\neq \emptyset $ whenever $%
F_{\alpha }$ is not nowhere dense, which is equivalent to saying that $%
F_{\alpha }$ is nowhere dense whenever $\overset{\circ }{F_{\alpha }^{\prime
}}=\emptyset .$ This proves (ii).

(ii) $\Rightarrow $ (iii). Obvious.

(iii) $\Rightarrow $ (iv). If $\alpha >m,$ then $F_{\alpha }$ is of second
category by definition of $m.$ It follows that $\overset{\circ }{F_{\alpha
}^{\prime }}\neq \emptyset ,$ for otherwise $F_{\alpha }$ is nowhere dense
by (iii), which is a contradiction. To see that $F_{\alpha }$ is nowhere
dense if $m>-\infty $ and $\alpha <m,$ observe that $F_{\alpha }^{\prime
}\subset F_{\beta }$ for any $\beta \in (\alpha ,m)$ and $F_{\beta }$ is of
first category by definition of $m.$ Thus, $F_{\alpha }^{\prime }$ is of
first category and so has empty interior. By (iii), $F_{\alpha }$ is nowhere
dense.

(iv) $\Rightarrow $ (v). Since $F_{\beta }^{\prime }\subset F_{\alpha }$ for
every $\beta <\alpha ,$ it follows from (iv) that $\overset{\circ }{
F_{\alpha }}\neq \emptyset $ when $\alpha >m.$

(v) $\Rightarrow $ (vi) Obvious.
\end{proof}

On face value, (vi) is significantly weaker than (v), even though they are
equivalent; see Theorem \ref{th13}. Also, it is readily checked that $%
F_{\alpha }^{\prime }$ can be replaced by $F_{\alpha }$ in (iii) (use (iii) $%
\Rightarrow $ (v)) but this is not true in (ii): It is possible that the set
of points of discontinuity of $f$ is of first category and $m>-\infty ,$ yet 
$F_{m}$ (always of first category, hence with empty interior) is not nowhere
dense:

\begin{example}
\label{ex3}Let $X$ be an infinite dimensional separable Banach space and let 
$(x_{n})_{n\in \Bbb{N}}\subset X$ be a dense sequence. Set $K_{n}:=\limfunc{%
span}\{x_{1},...,x_{n}\},$ so that $K:=\cup _{n\in \Bbb{N}}K_{n}$ is a dense
subspace of $X$ and $K$ is of first category since $\dim K_{n}<\infty .$
After passing to a subsequence, we may assume $K_{n}\varsubsetneq K_{n+1}$
without changing $K.$ Now, let $(\alpha _{n})\subset \Bbb{R}$ be a strictly
increasing sequence such that $\lim_{n\rightarrow \infty }\alpha _{n}=1.$
Set $f(x)=1$ if $x\in X\backslash K,f(x)=\alpha _{1}$ if $x\in K_{1}$ and $%
f(x)=\alpha _{n}$ if $x\in K_{n}\backslash K_{n-1},n\geq 2.$ Since $K$ is of
first category, $m=1.$ If $\alpha >1,$ then $F_{\alpha }=X$ has nonempty
interior while $F_{1}=K$ is convex of first category, but everywhere dense.
If $\alpha _{1}<\alpha <1,$ there is a unique $n\in \Bbb{N}$ such that $%
\alpha _{n}<\alpha \leq \alpha _{n+1}$ and so $F_{\alpha }=K_{n}$ is (convex
and) nowhere dense. If $\alpha \leq \alpha _{1},$ then $F_{\alpha
}=\emptyset .$ This shows that $f$ is quasiconvex and, by (i) $%
\Leftrightarrow $(v) in Corollary \ref{cor10}, that the set of points of
discontinuity of $f$ is of first category. (By a direct verification, this
set is $K.$)
\end{example}

Since $\overset{\circ }{F_{\alpha }}\neq \emptyset $ or $\overset{\circ }{
F_{\alpha }^{\prime }}\neq \emptyset $ is unchanged when $\alpha $ is
increased, Corollary \ref{cor10} makes it clear that everything depends upon
the behavior of $f$ below the critical level $m$ and ``just above'' it. In
practice, this means the following:

\begin{corollary}
\label{cor11}If $f:X\rightarrow \Bbb{R}$ is a quasiconvex function whose set
of points of discontinuity is of first category and if $g:X\rightarrow \Bbb{R%
}$ is another quasiconvex function such that $G_{\alpha }:=\{x\in
X:g(x)<\alpha \}$ (or $G_{\alpha }^{\prime }:=\{x\in X:g(x)\leq \alpha \}$)
coincides with $F_{\alpha }$ (or $F_{\alpha }^{\prime }$) for every $\alpha
<\alpha _{0}$ with $\alpha _{0}>m:=\mathcal{T}\limfunc{ess}\inf_{X}f,$ then
the set of points of discontinuity of $g$ is of first category.
\end{corollary}

\begin{proof}
Just notice that $m=\mathcal{T}\limfunc{ess}\inf_{X}g$ and use the
equivalence between (i) and (v) (or (iv)) for $f$ and for $g$ in Corollary 
\ref{cor10}.
\end{proof}

The hypotheses about $g$ in Corollary \ref{cor11} are equivalent to $g=f$ on 
$F_{\alpha _{0}}$ \emph{and} $f=g$ on $G_{\alpha _{0}}$ and therefore
stronger than either of these requirements alone. (Use $f(x)=\inf \{\alpha
\in \Bbb{R}:x\in F_{\alpha }^{\prime }\}$ and likewise for $g.$) The points
of discontinuity of $f$ and $g$ are of course the same inside $F_{\alpha
_{0}}=G_{\alpha _{0}}$ where $f$ and $g$ coincide, but unlike in Corollary 
\ref{cor7}, $f$ and $g$ may be completely different on large subsets, which
makes the result rather unexpected. There are various obvious (and perhaps
less obvious) ways to generalize Corollary \ref{cor11}.

\section{Special cases and related results\label{applications}}

The first two theorems in this section give especially simple conditions
ensuring that the set of points of discontinuity of a quasiconvex function
is of first category. The first one is convenient for the proof of Theorem 
\ref{th14} further below.

\begin{theorem}
\label{th12}Let $f:X\rightarrow \Bbb{R}$ be quasiconvex and set $m:=\mathcal{%
T}\limfunc{ess}\inf_{X}f.$ If $\overset{\circ }{F_{\alpha }^{\prime }}\neq
\emptyset $ (or $\overset{\circ }{F_{\alpha }}\neq \emptyset $) when $\alpha
>m$ and $\inf_{X}f=m$ (always true if $m=-\infty $), the set of points of
discontinuity of $f$ is of first category.
\end{theorem}

\begin{proof}
Since $\inf_{X}f=m$ means $F_{\alpha }=\emptyset $ when $\alpha <m,$ this
follows from the equivalence between (i) and (iv) (or (v)) in Corollary \ref
{cor10}.
\end{proof}

\begin{example}
If $f:X\rightarrow \Bbb{R}$ is quasiconvex, $f\geq 0,f(0)=0$ and $f$ is
continuous at $0,$ the set of points of continuity of $f$ is of first
category. Indeed, the continuity of $f$ at $0$ implies that $F_{\alpha
}^{\prime }$ ($F_{\alpha }$) is a neighborhood of $0$ for every $\alpha >0$
and (hence) that $m=0.$
\end{example}

The first part of the next theorem follows at once from Corollary \ref{cor10}%
, but more information is obtained.

\begin{theorem}
\label{th13}Let $f:X\rightarrow \Bbb{R}$ be quasiconvex and set $m:=\mathcal{%
T}\limfunc{ess}\inf_{X}f.$ If $\overset{\circ }{F_{\alpha }}\neq \emptyset $
when $\alpha >m$ and either $m=-\infty $ or $m>-\infty $ and $f^{-1}(m)$ is
of first category, the set of points of discontinuity of $f$ is of first
category. Furthermore, $F_{m}$ and $f^{-1}(m)$ (and hence also $%
F_{m}^{\prime }$) are nowhere dense.
\end{theorem}

\begin{proof}
Condition (vi) of Corollary \ref{cor10} is trivially satisfied, which
settles the issue about the points of discontinuity. To show that $F_{m}$ is
nowhere dense, observe that $F_{m}^{\prime }=f^{-1}(m)\cup F_{m}$ is of
first category (since $F_{m}$ is; see the proof of Corollary \ref{cor7}).
Thus, $\overset{\circ }{F_{m}^{\prime }}=\emptyset $ and the conclusion
follows from the equivalence between (i) and (ii) in Corollary \ref{cor10}.

It remains to show that $f^{-1}(m)$ is nowhere dense. Observe that if $x\in 
\overline{f^{-1}(m)},$ and $f(x)\neq m,$ then $f$ is not continuous at $x.$
Thus, $\overline{f^{-1}(m)}\subset f^{-1}(m)\cup A,$ where $A$ is the set of
points of discontinuity of $f.$ Since $A$ is of first category, it follows
that $\overline{f^{-1}(m)}$ is of first category. Therefore, it has empty
interior and $f^{-1}(m)$ is nowhere dense.
\end{proof}

If a function $f:X\rightarrow \Bbb{R}$ is continuous, its level sets $%
f^{-1}(\alpha )$ are closed. As a result, either they have nonempty interior
or they are nowhere dense. Theorem \ref{th14} below shows that, with at most
one exception, the same thing is true for quasiconvex functions whose set of
points of discontinuity is of first category. The proof is based primarily
on Theorem \ref{th12}. In that regard, see Remark \ref{rm3}.

\begin{theorem}
\label{th14}Suppose that $f:X\rightarrow \Bbb{R}$ is quasiconvex and that
the set of points of discontinuity of $f$ is of first category. Set $m:=%
\mathcal{T}\limfunc{ess}\inf_{X}f.$ If $\alpha \in \Bbb{R}$ and $\alpha \neq
m$ (always true if $m=-\infty $), then $f^{-1}(\alpha )$ is nowhere dense or
has nonempty interior.
\end{theorem}

\begin{proof}
If $m>-\infty $ and $\alpha <m,$ then $f^{-1}(\alpha )\subset F_{\beta }$
for any $\beta \in (\alpha ,m),$ whence $f^{-1}(\alpha )$ is nowhere dense
by the equivalence between (i) and (v) of Corollary \ref{cor10}. Thus, from
now on, $\alpha >m$ and it is assumed that $f^{-1}(\alpha )$ has empty
interior. The problem is to show that $f^{-1}(\alpha )$ is actually nowhere
dense.

In the proof of Theorem \ref{th13}, the fact that $f^{-1}(m)$ is nowhere
dense if it is of first category depends only on the set of points of
discontinuity of $f$ being of first category. Thus, the same thing is true
when $m$ is replaced by any other value $\alpha $ as soon as $f^{-1}(\alpha
) $ is of first category. Accordingly, it suffices to show that $%
f^{-1}(\alpha )$ is of first category when $\alpha >m$ and $f^{-1}(\alpha )$
has empty interior. This is done below.

If $h:\Bbb{R}\rightarrow \Bbb{R}$ is nondecreasing, $h\circ f$ is
quasiconvex. In particular, this is true of $h_{\alpha }^{-}\circ f$ and $%
h_{\alpha }^{+}\circ f,$ where $h_{\alpha }^{-}(t):=0$ if $0\leq t\leq
\alpha ,h_{\alpha }^{-}(t):=1$ if $t>\alpha $ and $h_{\alpha }^{+}(t):=0$ if 
$0\leq t<\alpha ,h_{\alpha }^{+}(t):=1$ if $t\geq \alpha .$ Below, we prove
that the sets of points of discontinuity of $h_{\alpha }^{-}\circ f$ and $%
h_{\alpha }^{+}\circ f$ are of first category and, next, that if $x\in
f^{-1}(\alpha ),$ at least one among $h_{\alpha }^{-}\circ f$ and $h_{\alpha
}^{+}\circ f$ is not continuous at $x.$ If so, $f^{-1}(\alpha )$ is
contained in the union of two sets of first category and the proof is
complete.

\textit{Step 1: }The set of points of discontinuity of $h_{\alpha }^{-}\circ
f$ is of first category.

If $a>0,$ the set $\Phi _{a}:=\{x\in X:h_{\alpha }^{-}\circ f<a\}$ contains
the set $F_{\alpha }$ and so it has nonempty interior since $\overset{\circ 
}{F_{\alpha }}\neq \emptyset $ by the equivalence of (i) and (v) of
Corollary \ref{cor10} (recall $\alpha >m$). In particular, $\Phi _{a}$ is of
second category, whence $\mathcal{T}\limfunc{ess}\inf_{X}(h_{\alpha
}^{-}\circ f)\leq a.$ Since $a>0$ is arbitrary, $\mathcal{T}\limfunc{ess}%
\inf_{X}(h_{\alpha }^{-}\circ f)\leq 0.$ Thus, $\mathcal{T}\limfunc{ess}
\inf_{X}(h_{\alpha }^{-}\circ f)=0$ since $h_{\alpha }^{-}\circ f\geq 0.$
Furthermore, $\inf_{X}(h_{\alpha }^{-}\circ f)=0$ since $0\leq
\inf_{X}(h_{\alpha }^{-}\circ f)\leq \mathcal{T}\limfunc{ess}
\inf_{X}(h_{\alpha }^{-}\circ f)=0.$ If follows that Theorem \ref{th12} is
applicable to $h_{\alpha }^{-}\circ f.$ This completes Step 1.

\textit{Step 2:} The set of points of discontinuity of $h_{\alpha }^{+}\circ
f$ is of first category.

All the arguments of Step 1 can be repeated verbatim.

\textit{Step 3:} If $x\in f^{-1}(\alpha ),$ then either $h_{\alpha
}^{-}\circ f$ or $h_{\alpha }^{+}\circ f$ is not continuous at $x.$

By contradiction, if both functions are continuous at $x,$ there is a
neighborhood $U$ of $x$ such that $U\subset (h_{\alpha }^{-}\circ
f)^{-1}(-1/2,1/2)$ (because $h_{\alpha }^{-}(\alpha )=0$) and $U\subset
(h_{\alpha }^{+}\circ f)^{-1}(1/2,3/2)$ (because $h_{\alpha }^{+}(\alpha )=1$
). Then, $h_{\alpha }^{-}(f(y))=0$ and $h_{\alpha }^{+}(f(y))=1$ for every $%
y\in U,$ i.e., $f(y)\leq \alpha $ and $f(y)\geq \alpha ,$ respectively, for
every $y\in U.$ Therefore $f=$ $\alpha $ on $U,$ which contradicts the
assumption that $f^{-1}(\alpha )$ has empty interior.
\end{proof}

\begin{remark}
\label{rm3} In Theorem \ref{th14}, suppose that $f$ is lsc (usc). Then, $%
h_{\alpha }^{+}\circ f$ ($h_{\alpha }^{-}\circ f$) in the proof need not be
lsc or usc. Therefore, even when $f$ is lsc or usc, the theorem does not
follow from the fact that the set of points of discontinuity of an lsc or
usc function is of first category. \newline
\end{remark}

If $f$ is a discontinuous linear form on $X,$ then $f^{-1}(\alpha )$ has
empty interior but is everywhere dense for every $\alpha \in \Bbb{R}.$ In
Example \ref{ex2}, $m=1$ and $f^{-1}(0)=H$ is of first category but also
everywhere dense. Thus, Theorem \ref{th14} breaks down quite easily when the
set of points of discontinuity of $f$ is not of first category.

\ If $X=\Bbb{R}^{N},$ then $\alpha =m$ need not be ruled out in Theorem \ref
{th14} because the set of points of discontinuity of a quasiconvex function
on $\Bbb{R}^{N}$ is always of first category (so, Step 1 and Step 2 of the
proof are unnecessary irrespective of $\alpha $). However, $f^{-1}(m)$ may
be of second category with empty interior when $\dim X=\infty ;$ see Example 
\ref{ex3}.

\begin{remark}
By using the $\sigma $-porosity result in \cite[Theorem 19]{BoWa05}, duly
extended to the non lsc case (Remark \ref{rm2}) in Step 3 of proof of
Theorem \ref{th14} but now with arbitrary $\alpha $ (everything else can be
ignored), it follows that if $X=\Bbb{R}^{N},$ every level set $f^{-1}(\alpha
)$ with empty interior is $\sigma $-porous.
\end{remark}

If $X$ is separable, $f^{-1}(\alpha )$ in Theorem \ref{th14} cannot have
nonempty interior for more than countably many values $\alpha \in \Bbb{R\ }$
since the sets $f^{-1}(\alpha )$ are disjoint. Elementary one-dimensional
examples show that countably many level sets $f^{-1}(\alpha )$ may indeed
have nonempty interior.

In Remark \ref{rm2}, we pointed out that a quasiconvex function on $\Bbb{R}
^{N}$ has the same points of (dis)continuity as its lsc hull $\underline{f}$
but that this is generally false when $\Bbb{R}^{N}$ is replaced by an
infinite dimensional tvs $X.$ Below, we show that this is still true if the
set of points of discontinuity of $f$ is of first category, but first a
minor a technicality must be clarified: Even when $f$ is real-valued, it may
happen that $\underline{f}=-\infty $ at some points. This is the reason why
the actual statement is a little more technical than the short version just
mentioned.

\begin{theorem}
\label{th15}Let $f:X\rightarrow \Bbb{R}$ be quasiconvex and let $\underline{f%
}$ denote its lsc hull. If the set $A$ of points of discontinuity of $f$ is
of first category, then $A=\{x\in X:\underline{f}(x)=-\infty \}\cup \{x\in X:%
\underline{f}(x)\in \Bbb{R},\underline{f}$ is discontinuous at $x\}.$
\end{theorem}

\begin{proof}
A routine verification shows that 
\begin{equation}
\underline{f}(x)=\inf \{\alpha \in \Bbb{R}:x\in \overline{F}_{\alpha }\}.
\label{8}
\end{equation}
We shall prove that $f$ is continuous at $x$ if and only if $\underline{f}
(x)\in \Bbb{R}$ and $\underline{f}$ is continuous at $x,$ which is
equivalent to the claim made in the theorem.

\textit{Step 1:} Assume that $f$ is continuous at $x.$

Given $\varepsilon >0,$ set $I_{\varepsilon }:(f(x)-\varepsilon
,f(x)+\varepsilon ),$ so that there is an open neighborhood $W_{\varepsilon
} $ of $x$ such that $f(W_{\varepsilon })\subset I_{\varepsilon }.$ Since $%
\underline{f}\leq f,$ it is obvious that $\underline{f}(y)<f(x)+\varepsilon $
for every $y\in W_{\varepsilon }.$ We claim that $\underline{f}(y)\geq
f(x)-\varepsilon $ for every $y\in W_{\varepsilon }.$ Otherwise, there is $%
y\in W_{\varepsilon }$ such that $\underline{f}(y)<f(x)-\varepsilon $ and so
there is $\alpha <f(x)-\varepsilon $ such that $y\in $ $\overline{F}_{\alpha
}.$ If so, $W_{\varepsilon }\cap F_{\alpha }\neq \emptyset $ since $%
W_{\varepsilon }$ is a neighborhood of $y,$ which is absurd since $\alpha
<f(x)-\varepsilon $ and $f(W_{\varepsilon })\subset I_{\varepsilon }.$ Thus, 
$\underline{f}(W_{\varepsilon })\subset [f(x)-\varepsilon ,f(x)+\varepsilon
),$ which show that $\underline{f}(x)\in \Bbb{R}$ (even $\underline{f}
(x)=f(x)$) and that $\underline{f}$ is continuous at $x.$

\textit{Step 2:} Assume $\underline{f}(x)\in \Bbb{R}$ and $\underline{f}$ is
continuous at $x.$

Let $\beta >\underline{f}(x)$ be given, so that $W:=\underline{f}%
^{-1}((-\infty ,$ $\beta ))$ is a neighborhood of $x.$ By definition of $%
\underline{f}$ in (\ref{8}), $W\subset \overline{F}_{\beta }$ and so $%
\overline{F}_{\beta }$ has nonempty interior containing $x.$ Note also that
if $y\in \overline{F}_{\beta }$ and $y\notin F_{\beta }^{\prime },$ then $f$
is not continuous at $y,$ so that $\overline{F}_{\beta }\subset F_{\beta
}^{\prime }\cup A$ where $A$ is the set of points of discontinuity of $f.$
Altogether, this yields $\overset{\circ }{\overline{F}_{\beta }}\subset
F_{\beta }^{\prime }\cup A.$ Since $A$ is of first category, then $B:=A\cap
(X\backslash F_{\beta }^{\prime })$ is of first category. Therefore, $%
\overset{\circ }{\overline{F}_{\beta }}\backslash B\subset F_{\beta
}^{\prime }$ and so, since $F_{\beta }^{\prime }$ is convex, it follows from
Lemma \ref{lm2} that $\overset{\circ }{\overline{F}_{\beta }}\subset
F_{\beta }^{\prime }.$ Since $x\in \overset{\circ }{\overline{F}_{\beta }},$
this shows that $F_{\beta }^{\prime }$ is a neighborhood of $x$ for every $%
\beta >\underline{f}(x)$ and (hence) that $f(x)=\underline{f}(x)$ since $%
\underline{f}\leq f.$

In summary, $F_{\beta }^{\prime }$ is a neighborhood of $x$ for every $\beta
>f(x).$ On the other hand, since $\underline{f}\leq f$ and $\underline{f}$
is continuous at $x,$ it follows that $f^{-1}([\alpha ,\infty ))$ $\supset 
\underline{f}^{-1}([\alpha ,\infty )),$ a neighborhood of $\underline{f}
(x)=f(x)$ for every $\alpha >f(x).$ Altogether, $f^{-1}([\alpha ,\beta ])$
is a neighborhood of $x$ for every $\alpha <f(x)<\beta ,$ whence $f$ is
continuous at $x.$
\end{proof}

\section{Complements\label{complements}}

We complete this paper with two complementary results of independent
interest. The first one clarifies the connection between the conditions for
the set of points of discontinuity of a quasiconvex function to be first
category (1) knowing that this is true for an equivalent quasiconvex
function (Corollary \ref{cor7}) and (2) without knowing that this is true
for an equivalent quasiconvex function (Sections \ref{continuity} and \ref
{applications}).

The second result was motivated by the conditions (ii) and (ii') of Lemma 
\ref{lm1}, that suggest a connection between quasiconvex and quasicontinuous
functions.

The sets $\overset{\circ }{F_{\alpha }}$ or $\overset{\circ }{F_{\alpha
}^{\prime }}$ are involved repeatedly in Sections \ref{continuity} and \ref
{applications}. For example, in Corollary \ref{cor10}, it is shown that the
set of points of discontinuity of $f$ is of first category if and only if $%
F_{\alpha }$ is nowhere dense whenever $\overset{\circ }{F_{\alpha }^{\prime
}}=\emptyset .$ Assuming this, Corollary \ref{cor7} gives a necessary and
sufficient condition for the set of points of continuity of another
quasiconvex function $g\sim _{\mathcal{T}}f$ to be of first category (after
exchanging the roles of $f$ and $g$ in the notation), but this condition
makes no implicit or explicit reference to $\overset{\circ }{F_{\alpha
}^{\prime }}$ or to its analog for $g.$ Yet, in a self-explanatory notation,
it must somehow imply that $G_{\alpha }$ is nowhere dense whenever $\overset{%
\circ }{G_{\alpha }^{\prime }}=\emptyset .$

This is clarified in Theorem \ref{th17} below, by showing that $\overset{%
\circ }{F_{\alpha }}$ and $\overset{\circ }{F_{\alpha }^{\prime }}$ are
unchanged after $f$ is replaced by any equivalent quasiconvex function. To
see this, we need a variant of Lemma \ref{lm2}.

\begin{lemma}
\label{lm16} Let $U\subset X$ be open and convex and let $A_{1},A_{2}\subset
X$ be of first category with $A_{2}\cap U=\emptyset $. If $G:=(U\backslash
A_{1})\cup A_{2}$ is convex, then $A_{1}\cap U=\emptyset $ and $\overset{%
\circ }{G}=U.$\newline
\end{lemma}

\begin{proof}
(i) The result is trivial if $U=\emptyset ,$ for then $G=A_{2}$ is of first
category, so that $\overset{\circ }{G}=\emptyset .$ From now on, $U\neq
\emptyset .$ We first prove that if $A_{1}\cap U\neq \emptyset ,$ there are $%
x,y\in U\backslash A_{1}\subset G$ such that $[x,y]\subset U$ contains a
point $z\in A_{1}.$ Since $A_{2}\cap (U\backslash A_{1})\subset A_{2}\cap
U=\emptyset ,$ it follows that $z\notin G,$ which contradicts the convexity
of $G.$

To prove the claim, suppose that $A_{1}\cap U\neq \emptyset .$ After
translation, it is not restrictive to assume $0\in A_{1}\cap U.$ Thus, $0\in
A_{1}\cup (-A_{1}),$ a set of first category. The set $V:=U\cap (-U)\subset
U $ is an open neighborhood of $0.$ Since $X$ is a Baire space, there is a
point $x\in V\backslash (A_{1}\cup (-A_{1})).$ Equivalently, both $x$ and $%
-x $ are in $V\backslash A_{1}\subset U\backslash A_{1}$ and their midpoint $%
0$ is in $A_{1}.$ This proves the claim with $y=-x.$

At this stage, we have shown that $A_{1}\cap U=\emptyset ,$ so that $%
U\backslash A_{1}=U$ and $G=U\cup A_{2}.$ Thus, $U\subset \overset{\circ }{G}
.$ Also, $\overset{\circ }{G}\subset G$ and so $\overset{\circ }{G}=U\cup
A_{3}$ with $A_{3}=A_{2}\cap \overset{\circ }{G}$ of first category. Since $%
A_{2}\cap U=\emptyset ,$ hence $A_{3}\cap U=\emptyset ,$ this also reads $U=%
\overset{\circ }{G}\backslash A_{3}.$ Therefore, by changing $U$ into $%
\overset{\circ }{G}$ (convex), $G$ into $U,A_{1}$ into $A_{3}$ and $A_{2}$
into $\emptyset $ above, it follows that $\overset{\circ }{G}\subset U$ and
so $\overset{\circ }{G}=U.$
\end{proof}

\begin{theorem}
\label{th17} If $f:X\rightarrow \overline{\Bbb{R}}$ is quasiconvex, then $%
\overset{\circ }{F_{\alpha }}$ and $\overset{\circ }{F_{\alpha }^{\prime }}$
are unchanged when $f$ is replaced by another \emph{quasiconvex} function $%
g\sim _{\mathcal{T}}f.$
\end{theorem}

\begin{proof}
We give the proof for $\overset{\circ }{F_{\alpha }}.$ Call $G_{\alpha }$
the set $\{x\in X:g(x)<\alpha \}.$ Below, we prove that if $\overset{\circ }{
F_{\alpha }}\neq \emptyset ,$ then $G_{\alpha }=(\overset{\circ }{F_{\alpha }%
}\backslash A_{1})\cup A_{2}$ where $A_{1}$ and $A_{2}$ are of first
category and $A_{2}\cap \overset{\circ }{F_{\alpha }}=\emptyset .$ From
Lemma \ref{lm16}, this implies $\overset{\circ }{G_{\alpha }}=\overset{\circ 
}{F_{\alpha }}$ and, by exchanging the roles of $f$ and $g,$ it follows that 
$\overset{\circ }{G_{\alpha }}=\overset{\circ }{F_{\alpha }}$ also when $%
\overset{\circ }{G_{\alpha }}\neq \emptyset .$ But then, $\overset{\circ }{
G_{\alpha }}=\overset{\circ }{F_{\alpha }}$ irrespective of whether either
is nonempty.

First, since $g\sim _{\mathcal{T}}f,$ it is clear that $F_{\alpha
}=(G_{\alpha }\backslash B_{1})\cup A_{1}$ where $A_{1}$ and $B_{1}$ are of
first category, $B_{1}\subset G_{\alpha }$ and $A_{1}\cap G_{\alpha
}=\emptyset .$ For future use, note that $B_{1}\cap F_{\alpha }=\emptyset $
since $A_{1}\cap B_{1}\subset A_{1}\cap G_{\alpha }=\emptyset .$ Next, since 
$\overset{\circ }{F_{\alpha }}\neq \emptyset ,$ observe that $F_{\alpha }=%
\overset{\circ }{F_{\alpha }}\cup B_{2}$ where $B_{2}\subset \partial
F_{\alpha }=\partial \overline{F}_{\alpha }$ (Remark \ref{rm1}) is nowhere
dense. Obviously, $B_{2}\cap \overset{\circ }{F_{\alpha }}=\emptyset .$

As a result, $(G_{\alpha }\backslash B_{1})\cup A_{1}=\overset{\circ }{
F_{\alpha }}\cup B_{2},$ both sides being $F_{\alpha }.$ Since $A_{1}\cap
(G_{\alpha }\backslash B_{1})\subset A_{1}\cap G_{\alpha }=\emptyset ,$ we
may take out $A_{1}$ from both sides without removing any point of $%
G_{\alpha }\backslash B_{1}$ and so $G_{\alpha }\backslash B_{1}=(\overset{
\circ }{F_{\alpha }}\cup B_{2})\backslash A_{1}=(\overset{\circ }{F_{\alpha }%
}\backslash A_{1})\cup (B_{2}\backslash A_{1}).$ Now, since $B_{1}\subset
G_{\alpha },$ it follows that $G_{\alpha }=(\overset{\circ }{F_{\alpha }}
\backslash A_{1})\cup A_{2}$ where $A_{2}:=(B_{2}\backslash A_{1})\cup
B_{1}. $ Both $A_{1}$ and $A_{2}$ are of first category and $A_{2}\cap 
\overset{\circ }{F_{\alpha }}=\emptyset $ since $B_{2}\cap \overset{\circ }{
F_{\alpha }}=\emptyset $ and $B_{1}\cap F_{\alpha }=\emptyset .$
\end{proof}

If ``second category'' is replaced by ``nonempty interior'' in condition
(ii) ((ii')) of Lemma \ref{lm1}, the modified condition is exactly the
assumption that the function $f$ of interest is what is called \emph{\ lower
quasicontinuous }(\emph{upper quasicontinuous}). See \cite{EwLi83}, \cite
{Ne88}, \cite{Sa93}, among others.

The stronger requirement that for every open subset $U\subset X$ and every
open subset $\Omega \subset \Bbb{R}$ such that $f(U)\cap \Omega \neq
\emptyset ,$ the set $U\cap f^{-1}(\Omega )$ has nonempty interior is called 
\emph{quasicontinuity,} a concept introduced by Kempisty \cite{Ke32} in
1932, which has been vigorously revisited in recent past. It has nothing to
do with its namesake occasionally used in convex analysis (\cite{JoLa71}, 
\cite{MoVo97}) and makes sense when the target space $\Bbb{R}$ is replaced
by any topological space $Y.$

A well known shorter equivalent definition is that $f$ is quasicontinuous if
and only if, for every open subset $\Omega \subset \Bbb{R},$ the interior of 
$f^{-1}(\Omega )$ is dense in $f^{-1}(\Omega ).$ Shi \textit{et al.} \cite
{ShZhZh95} used this definition to rediscover quasicontinuous functions six
decades later under the name ``robust functions'', but this terminology is
not prevalent.

Every continuous function is quasicontinuous but quasicontinuity and
semicontinuity are very different properties: With $X=\Bbb{R},$ the function 
$f=\chi _{\Bbb{R}\backslash \{0\}}$ is lsc but not quasicontinuous, whereas $%
f(x)=\sin (x^{-1})$ if $x\neq 0$ and $f(0)=0$ is quasicontinuous but neither
lsc nor usc.

At any rate, the resemblance noted in Lemma \ref{lm1} raises the natural
question whether a real-valued quasiconvex function is quasicontinuous. The
answer is positive in the usc case, as we now show. The Baire property is
not needed.

\begin{theorem}
\label{th18}Let $X$ be a tvs. If $f:X\rightarrow \Bbb{R}$ is quasiconvex and
usc, then $f$ is quasicontinuous.
\end{theorem}

\begin{proof}
Let $\Omega \subset \Bbb{R}$ be open. It must be shown that the interior of $%
f^{-1}(\Omega )$ is dense in $f^{-1}(\Omega ).$ A routine verification
reveals that this is true if and only if it is true when $\Omega =(\alpha
,\beta )$ is a finite interval such that $f^{-1}(\Omega )\neq \emptyset .$

It is plain that $f^{-1}(\Omega )=F_{\beta }\backslash F_{\alpha }^{^{\prime
}}.$ Since $f$ is usc and quasiconvex, $F_{\beta }$ is open and convex and,
by the openness of $F_{\beta },$ the interior of $f^{-1}(\Omega )$ is $%
F_{\beta }\backslash \overline{F}_{\alpha }^{\prime }.$ Let $W$ denote the
interior of $\overline{F}_{\alpha }^{\prime }.$ We begin by showing that $W=%
\overset{\circ }{F_{\alpha }^{\prime }}.$ Since this is obvious if $%
W=\emptyset ,$ we henceforth assume that $W\neq \emptyset .$ If $\dim
X=\infty ,$ it is generally false that a convex subset has the same interior
as its closure (this does \emph{not} follow from Remark \ref{rm1}), so a
proof is needed. That $f$ is usc is important.

First, $F_{\alpha }^{\prime }=\cap _{\gamma >\alpha }F_{\gamma },$ so that $%
\overline{F}_{\alpha }^{\prime }\subset \cap _{\gamma >\alpha }\overline{F}
_{\gamma }.$ It follows that $W\subset \overline{F}_{\gamma }$ for every $%
\gamma >\alpha .$ Furthermore, $F_{\gamma }$ (open since $f$ is usc) is the
interior of $\overline{F}_{\gamma }$ (Remark \ref{rm1}) and so $W\subset
F_{\gamma }.$ As a result, $W\subset \cap _{\gamma >\alpha }F_{\gamma
}=F_{\alpha }^{^{\prime }},$ so that $W$ is contained in $\overset{\circ }{
F_{\alpha }^{\prime }}.$ By definition of $W,$ the converse is trivial, so
that $W=\overset{\circ }{F_{\alpha }^{\prime }},$ as claimed.

If $x\in F_{\beta }\backslash F_{\alpha }^{^{\prime }}$ ($=f^{-1}(\Omega
)\neq \emptyset $ ), then $x\notin \overset{\circ }{F_{\alpha }^{\prime }}
=W. $ Since $W$ is the interior of $\overline{F}_{\alpha }^{\prime },$ every
open neighborhood of $x$ in $X$ contains a point not in $\overline{F}
_{\alpha }^{\prime }.$ In particular, if $U$ is an open neighborhood of $x$
in $X,$ then $U\cap F_{\beta }$ (open) contains a point $y\notin \overline{F}
_{\alpha }^{\prime }.$ Hence, $U$ contains a point of $F_{\beta }\backslash 
\overline{F}_{\alpha }^{\prime },$ so that $F_{\beta }\backslash \overline{F}
_{\alpha }^{\prime }$ is dense in $F_{\beta }\backslash F_{\alpha }^{\prime
}.$
\end{proof}

The example when $X=\Bbb{R},f(0)=-1$ and $f(x)=|x|$ when $x\neq 0$ shows
that $f$ need not be quasicontinuous if it is not usc.

\end{document}